\documentclass[reqno,10pt]{article}
\usepackage[a4paper,
  left=18mm,
  right=18mm,
  top=18mm,
  bottom=22mm]{geometry}
  \usepackage[T1]{fontenc}
\usepackage{lmodern}
\usepackage{mathtools}
\usepackage{amsmath,amssymb,amsthm,mathrsfs}
\usepackage{geometry}
\usepackage{hyperref}
\usepackage{enumitem}
\usepackage{microtype} 
\hypersetup{colorlinks=true,linkcolor=blue,citecolor=blue,urlcolor=blue}

\theoremstyle{plain}
\newtheorem{theorem}{Theorem}[section]
\newtheorem{proposition}[theorem]{Proposition}
\newtheorem{lemma}[theorem]{Lemma}
\newtheorem{corollary}[theorem]{Corollary}

\theoremstyle{definition}
\newtheorem{definition}[theorem]{Definition}
\newtheorem{remark}[theorem]{Remark}

\newcommand{\T}{\mathbb{T}}

\newcommand{\Hs}{\mathcal{H}}

\def\bega{\begin{aligned}}
\def\enda{\end{aligned}}
\def\R{\mathbb{R}^2}

\def\R{\mathbb{R}}

\def\bcase{\begin{cases}}
\def\ecase{\end{cases}}

\def\bmx{\begin{bmatrix}}
\def\emx{\end{bmatrix}}

\def\bitem{\begin{itemize}}
\def\eitem{\end{itemize}}

\def\w{\omega}

\def\beq{\begin{equation}}
\def\eeq{\end{equation}}
\def\dist{\text{dist}}

\title{Regularization for Multi-Phase 2D Euler Equations via Competing Transport Markers}

\author{%
Trinh T. Nguyen\\[0.3em]
\small School of Mathematical Sciences, Shanghai Jiao Tong University\\
\small \texttt{q0004349@sjtu.edu.cn}
}

\date{}

\begin{document}
\maketitle
\begin{abstract}
We introduce a novel regularization framework for the two-dimensional incompressible Euler equation that exactly preserves the transport structure of multi-phase vorticity fields. The key step is a reformulation of multi-phase vortex patch data in terms of a finite family of passively advected scalar marker functions: at each point, the local vorticity is determined by a smooth, pointwise selection rule arising from competition among these markers. The scheme introduces no spatial diffusion or mollification; all regularization originates solely from the marker selection mechanism. As the sharpness parameter $\beta\to\infty$, we prove uniform convergence of the transported marker functions on finite time intervals. Moreover, under a geometric nondegeneracy condition on the underlying Euler interface network, we establish Hausdorff convergence of the evolving interfacial structures and exponential-in-$\beta$ pointwise convergence of the regularized vorticity away from the tie sets to the corresponding sharp multi-phase vortex patch solution. Finally, we show that the loss of pointwise convergence coincides precisely with the onset of geometric degeneracy in the Euler interface dynamics.\end{abstract}


\section{Introduction}
\noindent
The evolution of piecewise--constant vorticity fields, commonly referred to as
\emph{vortex patches}, is a classical and central topic in the analysis of the
two--dimensional incompressible Euler equation.
In its simplest form, a vortex patch consists of a region of constant vorticity
separated from its complement by a sharp interface.
Beyond their mathematical significance, such patch--type configurations serve as
idealized models for coherent vortical structures observed in a wide range of
inviscid flows
\cite{ConstantinFoias1988,Burbea1982,MarchioroPulvirenti1994,MarchioroPulvirenti1993,Saffman1992,Lamb1932,childress2009,DrivasElgindiJeong2024,DrivasElgindi2023EMS,BedrossianVicol2022}. More complex dynamics arise when the vorticity field consists of several distinct
patches with different vorticity levels.
Long-lived cellular vortex configurations have been observed in both numerical
simulations and laboratory experiments of nearly inviscid two-dimensional flows.
In particular, experiments have demonstrated the spontaneous formation of
so-called \emph{vortex crystals}, consisting of a finite number of coherent vortices
arranged in regular lattice-like patterns and persisting over extremely long time
intervals \cite{PhysRevLett.75.3277}.
These states arise dynamically from turbulent initial conditions and are
characterized by sharply separated regions of nearly uniform vorticity embedded
in a weaker background. Such observations motivate the mathematical study of multi-phase vorticity
configurations and the evolution of their interfaces within the incompressible
Euler framework
\cite{SiegelmanYoungIngersoll2022,hussain1986,dileoni2020,SchecterDubin1999,ArefNewtonStremlerTokiedaVainchtein2003}.~\\~\\
The dynamics of vortex patches with regular boundaries have been extensively
investigated over the past decades, beginning with the foundational works
\cite{BertozziConstantin1993,Chemin93,Chemin95} and continuing through a broad
literature on contour dynamics and boundary regularity.
See also
\cite{buttke1989, HmidiMateuVerdera2013,ElgindiJeong2023,KiselevLuo2023,ZabuskyHughesRoberts1979,DrivasElgindiLa2023,elgindi2025cuspformationvortexpatches,park2022uniform,ElgindiHuang2025BahouriChemin,HuangZhang2025HolderBahouriChemin}
and the references therein. These results establish that when the initial vorticity consists of one or several
patches with $C^{1,\alpha}$ boundaries, the interfaces are transported by the flow
and remain regular for all time within the Yudovich class.
The evolution may therefore be described entirely in terms of the motion of the
patch boundaries, leading to a closed contour--dynamics formulation.
In this setting, the vorticity is represented as a finite sum of characteristic
functions whose supports evolve under the Euler flow, providing a precise description of interface dynamics.~\\~\\
On the other hand, a complementary challenge arises when one seeks to
\emph{regularize} interface problems—whether for analytical approximation,
numerical computation, or for establishing rigorous connections between inviscid
dynamics and more microscopic descriptions.
Such lack-of-regularity issues arise prominently in aspects of Hilbert’s sixth
problem \cite{hilbert1900problems,Saint-Raymond-book,bardos-golse-levermore1,bardos-golse-levermore2,bardos-golse-levermore4,bardos-golse-levermore3}, which concerns the passage from kinetic or molecular models to continuum fluid equations. In this context, Hilbert expansions and related asymptotic procedures necessitate regularization frameworks that respect the underlying geometric structure of the flow \cite{guo2006boltzmann,JangKimAPDE,YuCPAM,Speck-Strain,KimNguyen,KimNguyen2025Prandtl}.
Standard regularization techniques, such as convolution with a mollifier,
typically introduce commutator errors that disrupt the intrinsic transport
structure of the Euler equation: the regularized vorticity is no longer expressed
in terms of advected quantities, and the resulting smoothing is not aligned with
the evolving interfaces.
As a consequence, the regularized dynamics become difficult to relate to the
original interface motion, and quantitative control of sharp--interface limits
becomes unclear. A different approach to regularization proceeds through the
vanishing--viscosity limit of the Navier--Stokes equations.
Substantial progress has been made along these lines, including rigorous
derivations of inviscid dynamics and stability results for Yudovich solutions and
related configurations; see, for instance,
\cite{ConstantinWu95,ConstantinWu96,Chemin96,Masmoudi07,Sueur15,PeterTarekTheo}
and the references therein.
Closely related vanishing--viscosity and singular--limit results have also been
obtained for other vortex models, including point vortices
\cite{MarchioroPulvirenti1993,Marchioro1990OnTV,CeciSeis2024,Gallay2010},
vortex--wave systems \cite{2NVW,bjorland2011}, and vortex rings
\cite{benedetto2000,brunelli2011,cao2021,gallay_sverak2024}.
From a physical perspective, diffusion--based regularization is both natural and
fundamental, and its analysis lies at the heart of singular perturbation
theory in fluid mechanics
\cite{prandtl1904,lagerstrom1972,wood1957,oleinik1999,kim_childress2001,2N,TN-critical,2N-ext,DrivasIyerNguyen2024FeynmanLagerstrom,BNTT,Toan-Grenier-APDE,Grenier-CPAM,GN2,GreGuoToanAM,GGN3,SammartinoCaflisch2,GVMM,Mae,fei2023prandtl,KVW,KNVW,ConstantinKuka,Wang,Maz4}. However, an additional layer of subtlety arises in \emph{multi--patch
configurations}, where more than two vorticity levels are present.
In such settings, the flow decomposes into several regions separated by a network
of interfaces that may form cellular patterns and, in principle, meet at
junctions.
While viscosity regularizes the equations, it simultaneously smooths the 
interfaces and corners that characterize the inviscid geometry. 
Consequently, even when convergence to an inviscid solution can be established in $L^p$, identifying the limiting evolution of the
interface network remains unclear.
Such multi--interface configurations arise naturally as coarse--grained
descriptions of organized vorticity fields on periodic domains and in idealized
models of vortex crystals or coherent states. ~\\~\\
A natural and intriguing question is whether there exists an
\emph{interpolating viewpoint} between these two perspectives—one in which
interface geometry remains meaningful while the vorticity is described by a
regularized, smooth field.
\paragraph{Question.}
Can one construct a regularization of the two--dimensional incompressible Euler
equation in genuinely multi--phase vorticity configurations that preserves the
exact transport structure of the inviscid flow and admits a coherent geometric
description of evolving interfaces and junctions? More precisely, can multi--phase
vortex patch data be regularized in a manner that respects the underlying phase
partition and faithfully tracks the transported interface geometry, and can the
breakdown of pointwise convergence of such a regularization be characterized
purely in geometric terms through degenerations of the Euler interface network?~\\~\\
In this paper, we propose a regularization of the two--dimensional incompressible
Euler equation that operates directly in physical space and is tailored to
genuinely multi--phase vorticity configurations. The construction introduces no
spatial mollification or diffusion. Instead of prescribing the vorticity through
a fixed partition of the domain and explicitly tracking the evolution of its
interfaces, we introduce a finite family of scalar markers that are
passively transported by the flow, and determine the local vorticity through a
smooth gating rule that produces a convex mixture of vorticity phases, with
weights depending on the relative magnitudes of the markers. The regularization is governed by a sharpness parameter $\beta>0$. A central
feature of the construction is a structural invariance: for each fixed $\beta$,
the vorticity $\w^\beta$ solves the Euler equation and retains for all time the functional form
\[
\omega^\beta(t,\cdot)
=
F_\beta\bigl(\varphi_1^\beta(t,\cdot),\dots,\varphi_K^\beta(t,\cdot)\bigr),
\qquad
(\partial_t+u^\beta\cdot\nabla)\varphi_k^\beta=0,
\]
so that all diffuse interfacial geometry is encoded entirely by scalar quantities
transported by the flow. This exact transport representation is preserved
throughout the evolution and is precisely the structure that is typically
destroyed by standard regularization procedures, such as convolution with a
mollifier. For suitable choices of the marker functions, the construction yields a
$C^\infty$ regularization of cellular vorticity patterns, including those arising
in idealized models of vortex crystals.
\paragraph{Organization of the paper.}
In Section~\ref{multi}, we formulate the two--dimensional incompressible Euler
equation in vorticity form \eqref{eq:euler} and introduce multi--phase vortex patch
data \eqref{ran-4}. We then reformulate the multi--phase patch data in a form
adapted to Euler transport through a \emph{competitive--marker viewpoint}: a finite
family of marker functions $\{\varphi_k^{\mathrm{in}}\}_{k=1}^K$ is transported by
the flow and induces, via pointwise maximization, a partition of the domain into
regions $\Omega_k$, with interfaces described by the associated tie sets. This
reformulation encodes the evolving multi--phase geometry directly in terms of
transported scalar fields. Based on this representation, we define a regularized
initial vorticity as a convex mixture of phases \eqref{initial}. In Section~\ref{const}, we establish closure of the regularized ansatz under Euler
evolution (Theorem~\ref{close1} and Proposition~\ref{conserve}). In
Section~\ref{mark-conv}, we formulate geometric nondegeneracy assumptions on the tie
sets and prove their persistence under transport. We also derive an $L^1$
approximation of the initial patch data by the regularized vorticity
$\omega_0^\beta$ (Proposition~\ref{prop:expG_patch_approx}) and prove uniform
convergence of the transported marker functions on any finite time interval
$[0,T]$ (Theorem~\ref{conv-of-markers}). In Section~\ref{sec-hau}, we study the Hausdorff convergence of the evolving
interfaces, establishing convergence up to a maximal time $T_\star$, defined as
the first time at which the Euler interface network becomes geometrically
degenerate (Theorems~\ref{hau1} and~\ref{hau2}). Finally, we derive sharp pointwise
vorticity estimates away from the tie sets, with errors that decay exponentially in
$\beta$ (Theorems~\ref{1patch} and~\ref{2patch}), and show that the loss of pointwise
convergence coincides precisely with the onset of singular geometry in the
underlying Euler interface dynamics.
\section{Multi-Phase Partitions: A Competitive-Marker Viewpoint}\label{multi}
We consider the incompressible Euler equation in vorticity form
\begin{equation}\label{eq:euler}
\partial_t \omega + u\cdot \nabla \omega = 0,
\qquad
u = \nabla^\perp \psi,
\qquad
-\Delta \psi = \omega - \langle \omega \rangle,
\end{equation}
posed on the two-dimensional torus $\T^2$. Here $\langle\omega\rangle$ denotes the spatial mean of $\omega$.
The initial vorticity is defined by
\beq\label{ran-4}
\omega_0(x)
=
\sum_{k=1}^K c_k\,\mathbf{1}_{\Omega_k}(x),
\eeq
where $c_k\in\R$ denotes the vorticity level assigned to the phase
$\Omega_k\subset \T^2$, and $K\ge 2$ is a natural number.
Yudovich theory yields global existence and uniqueness of solutions in the Yudovich class \cite{yudovich1963,Yudovich1995,majda2002}.
\paragraph{Vorticity phases as competing regions.}
Rather than prescribing the vorticity through a fixed partition of the domain and
explicitly tracking the subsequent motion of its interfaces, we adopt a different
geometric viewpoint. We represent the domain by a finite family of \emph{competing
regions} $\Omega_k$, each associated with a prescribed vorticity level and
determined pointwise by a collection of scalar marker functions. At each spatial
location, the vorticity is selected according to the dominant marker value, so that
the markers encode the local prevalence of the corresponding phase.
 In this formulation, interfacial boundaries are not imposed \emph{a priori} but
emerge dynamically as the loci where two or more marker functions are in balance;
we refer to these sets as \emph{tie sets}. Since the marker functions are transported
by the Euler flow, the induced phase partition and its interfaces evolve coherently
with the underlying fluid motion. This reformulation replaces explicit interface
tracking with a competition-based description in which interfacial geometry is
encoded implicitly through transported scalar fields.
\paragraph{Domain partitions via markers.}
Let $K\ge 2$ and $\alpha\in(0,1)$, and let
$\{\varphi_k^{\mathrm{in}}\}_{k=1}^K \subset C^{1,\alpha}(\T^2)$ be a collection of
scalar functions, which we refer to as \emph{markers}.
At each point $x\in\T^2$, the values
$\varphi_1^{\mathrm{in}}(x),\dots,\varphi_K^{\mathrm{in}}(x)$ are compared, and $x$ is
assigned to the phase $\Omega_k$ corresponding to the largest marker value.
Accordingly, the marker functions may be chosen so that the torus is partitioned
into $K$ disjoint open regions
\begin{equation}\label{domain}
\Omega_k
:=
\bigl\{x\in\T^2:\ \varphi_k^{\mathrm{in}}(x)>\varphi_j^{\mathrm{in}}(x)
\ \text{for all } j\neq k\bigr\},
\qquad k=1,\dots,K,
\end{equation}
whose union covers $\T^2$ up to a set of measure zero. In this representation, each phase $\Omega_k$ is characterized by the dominance of
the corresponding marker $\varphi_k^{\mathrm{in}}$. The boundaries between phases
are given by the \emph{tie sets}
\[
\Gamma_{ij}
:=
\bigl\{x\in\T^2:\ \varphi_i^{\mathrm{in}}(x)=\varphi_j^{\mathrm{in}}(x)\ge
\varphi_\ell^{\mathrm{in}}(x)\ \text{for all }\ell\bigr\},
\]
which form a network of interfacial curves separating the regions.
When $K\ge 3$, these interfaces may meet at junctions where three or more marker
values coincide.
\paragraph{Vorticity regularization as a competitive weighting.}
Given a sharpness parameter $\beta>0$, we define the weights
\begin{equation}\label{eq:expG}
\pi_{k}^\beta(x)=\frac{e^{\beta \varphi_{k}^{\text{in}}(x)}}{\sum_{j=1}^K e^{\beta \varphi_j^{\text{in}}(x)}},
\qquad k=1,\dots,K.
\end{equation}
We regularize the initial vorticity \eqref{ran-4} by
\beq\label{initial}
\omega_0^\beta(x)=\sum_{k=1}^K c_k \,\pi_k^\beta(x).
\eeq
This construction admits a \textbf{mixture of colors} interpretation.
The constants $\{c_k\}_{k=1}^K$ may be viewed as a finite palette of colors (or
vorticity levels), while the markers $\{\varphi_k^{\text{in}}\}_{k=1}^K$ encode, at each
location $x\in\T^2$, how strongly the point ``prefers'' each color.
Given the marker scores $\{\varphi_1^{\text{in}}(x),\dots,\varphi_K^{\text{in}}(x)\}$, the point $x$ then forms a
\textbf{convex mixture} of the available colors according to the weights
\eqref{eq:expG}.
Equivalently, $x$ assigns probability $\pi_k^\beta(x)$ to the color $c_k$ and takes
the averaged vorticity value
\[
\omega_0^\beta(x)=\sum_{k=1}^K c_k\,\pi_k^\beta(x).
\]
The parameter $\beta>0$ controls how decisive this choice is.
Indeed, for any $k,\ell$,
\[
\frac{\pi_k^\beta(x)}{\pi_\ell^\beta(x)}
=
\exp\!\bigl(\beta(\varphi_k^{\text{in}}(x)-\varphi_\ell^{\text{in}}(x))\bigr),
\]
so ratios in marker scores are exponentially amplified as $\beta$ increases.
Thus, when one score is strictly larger than the others, the corresponding weight
approaches $1$ while the remaining weights become exponentially small.
In the singular limit $\beta\to\infty$, this mixture of colors collapses to a sharp
coloring: each point $x$ takes the single color $c_k$ whose score $\varphi_k^{\text{in}}(x)$ is
maximal, and transitions between colors are confined to thin layers near the tie
sets where two or more scores coincide. This initial regularized configuration incorporates the exact interface geometry into the Euler equations.
\paragraph{Comparison with mollifier regularization.}
\begin{itemize}
\item
A standard approach to regularizing vortex patch data is to mollify the initial
vorticity,
\[
\omega_0^\varepsilon := \rho_\varepsilon * \omega_0,
\]
and to solve the Euler equation with initial data $\omega_0^\varepsilon$.
This procedure amounts to a nonlocal Eulerian smoothing: each point is replaced by
an average of nearby values, independently of whether it lies deep inside a phase
or near an interface. As a consequence, sharp boundaries are uniformly smeared, and
the regularized vorticity no longer admits a description in terms of transported
level sets. In particular, the smoothing is not aligned with the evolving interface
geometry.

By contrast, in the present construction \eqref{initial}, the marker functions
$\{\varphi_k^{\mathrm{in}}\}$ are transported exactly by the Euler flow. Regularization
occurs solely through pointwise competition among the markers and is therefore
localized near regions where two or more phases interact. Away from interfaces, a
single marker dominates and the vorticity remains essentially constant. As a result,
the diffuse transition layers are confined to thin neighborhoods of transported tie
sets and remain geometrically aligned with the flow for all times.

\item
If the initial marker functions $\{\varphi_k^{\mathrm{in}}\}$ belong to
$C^{1,\alpha}(\T^2)$, then the regularized initial vorticity \eqref{initial}
inherits the same $C^{1,\alpha}$ regularity. In this sense, the construction yields
a \emph{$C^{1,\alpha}$--consistent regularization} of multi--phase vortex patch data,
compatible with the classical Euler well-posedness theory
\cite{Lichtenstein1925,Gunther1927}, rather than imposing artificial $C^\infty$
smoothing. If the markers are chosen smooth, the scheme naturally produces a
$C^\infty$ regularization.

\item
Unlike boundary--based regularizations, which become non-canonical in genuinely
multi--phase configurations involving junctions or intersecting interfaces, the
marker formulation regularizes all phases simultaneously through a single global
representation. This avoids additional compatibility conditions at junctions and
preserves the geometric coherence of the entire interface network.
Moreover, the resulting vorticity converges pointwise away from the tie sets, with
errors that decay exponentially in the sharpness parameter $\beta$.
\end{itemize}

\section{Construction of solutions}\label{const}
\begin{theorem}\label{close1}
Let $(\omega^\beta,u^\beta)$ be the solution of the Euler equation
\eqref{eq:euler} with initial data \eqref{initial}. For each $k\in\{1,\dots,K\}$,
let $\varphi_k^\beta$ be the solution of the transport equation with the common
velocity field $u^\beta$,
\begin{equation}\label{ran-2}
\partial_t \varphi_k^\beta + u^\beta\cdot\nabla \varphi_k^\beta = 0,
\end{equation}
with initial condition
\[
\varphi_k^\beta\big|_{t=0} = \varphi_k.
\]
Then the vorticity $\omega^\beta$ admits the representation
\begin{equation}\label{ran-1}
\omega^\beta(t,x)
=
\sum_{k=1}^K c_k\,\pi_k^\beta(t,x),
\end{equation}
where the coefficients $\pi_k^\beta$ are given by
\[
\pi_{k}^\beta(t,x)
=
\frac{e^{\beta \varphi_{k}^\beta(t,x)}}
{\sum_{j=1}^K e^{\beta \varphi_j^\beta(t,x)}},
\qquad k=1,\dots,K.
\]
\end{theorem}
\begin{proof}
Define
\[
\widetilde{\omega}(t,x)
:=
\sum_{k=1}^K c_k\,\pi_k^\beta(t,x),
\qquad
\pi_k^\beta(t,x)
=
\frac{e^{\beta\varphi_k^\beta(t,x)}}{\sum_{j=1}^K e^{\beta\varphi_j^\beta(t,x)}} .
\]
We show that $\widetilde{\omega}$ satisfies the same transport equation and the
same initial condition as $\omega^\beta$ for the given velocity field $u^\beta$.
Since each score field $\varphi_k^\beta$ solves \eqref{ran-2}, 
the chain rule gives
\[
\partial_t e^{\beta\varphi_k^\beta}
+u^\beta\cdot\nabla e^{\beta\varphi_k^\beta}
=
\beta e^{\beta\varphi_k^\beta}
\bigl(\partial_t\varphi_k^\beta+u^\beta\cdot\nabla\varphi_k^\beta\bigr)
=0.
\]
Summing over $k$, the denominator
\[
S(t,x):=\sum_{j=1}^K e^{\beta\varphi_j^\beta(t,x)}
\]
satisfies
\[
\partial_t S+u^\beta\cdot\nabla S=0.
\]
Applying the quotient rule, we obtain
\[
(\partial_t+u^\beta\cdot\nabla)\pi_k^\beta
=
\frac{(\partial_t+u^\beta\cdot\nabla)e^{\beta\varphi_k^\beta}}{S}
-
\frac{e^{\beta\varphi_k^\beta}}{S^2}
(\partial_t+u^\beta\cdot\nabla)S
=0,
\qquad k=1,\dots,K.
\]
Thus each weight $\pi_k^\beta$ is transported by the velocity field $u^\beta$.
By linearity,
\[
(\partial_t+u^\beta\cdot\nabla)\widetilde{\omega}
=
\sum_{k=1}^K \omega_k\,(\partial_t+u^\beta\cdot\nabla)\pi_k^\beta
=0.
\]
At $t=0$ we have $\varphi_k^\beta(0,\cdot)=\varphi_k(\cdot)$, hence
\[
\pi_k^\beta(0,x)=\pi_k^\beta(x),
\]
and therefore
\[
\widetilde{\omega}(0,x)
=
\sum_{k=1}^Kc_k\,\pi_k^\beta(x)
=
\omega_0^\beta(x).
\]
Thus $\omega^\beta$ and $\widetilde{\omega}$ solve the same linear transport
equation with the same velocity field $u^\beta$ and the same initial data.
By uniqueness of bounded solutions to the transport equation under Yudovich
(log-Lipschitz) velocities, we conclude that
\[
\omega^\beta(t,x)=\widetilde{\omega}(t,x)
\quad\text{for all }t\in[0,T],\ x\in\T^2,
\]
which proves \eqref{ran-1}.
\end{proof}

\begin{remark}[Structural embedding via transported markers]
More generally, the Euler dynamics admits a structural embedding in terms of
transported scalar markers. Specifically, multi--phase vorticity
configurations can be realized as nonlinear observables of a finite family of
passively transported functions, whose relative ordering encodes the phase
geometry. In this representation, all interfacial structure is implicitly
determined by competition between transported quantities, rather than by
explicit interface tracking. We record this observation here as it provides a
useful geometric reparameterization of multiphase Euler flows and appears to be
of independent interest.
\end{remark}

\begin{proposition}[Closure of the multiphase ansatz]
\label{conserve}
Let $K\ge 2$ be a natural number, and let
$\{\varphi_k\}_{k=1}^K \subset C^{1,\alpha}(\T^2)$ be given scalar marker functions.
Assume that $F:\R^K\to\R$ is a smooth function, and define the initial vorticity by
\[
\omega_0(x)=F\bigl(\varphi_1^{\text{in}}(x),\dots,\varphi_K^{\text{in}}(x)\bigr).
\]
Let $(\omega,u)$ be the Yudovich solution of the two--dimensional Euler equation
\begin{equation}\label{eq:euler-closure}
\partial_t\omega + u\cdot\nabla\omega = 0,
\qquad
u = K\ast\omega,
\end{equation}
with initial data $\omega_0$.
Then the vorticity admits the representation
\[
\omega(t,x)
=
F\bigl(\varphi_1(t,x),\dots,\varphi_K(t,x)\bigr),
\]
where each $\varphi_k$ solves the transport equation
\[
\partial_t\varphi_k + u\cdot\nabla\varphi_k = 0,
\qquad
\varphi_k|_{t=0}=\varphi_k^{\text{in}}.
\]
\end{proposition}

\begin{proof}
Let
\[
\widetilde{\omega}(t,x)
:=
F\bigl(\varphi_1(t,x),\dots,\varphi_K(t,x)\bigr),
\]
where $\{\varphi_k\}$ are transported by the velocity field $u$.
By the chain rule,
\[
\partial_t \widetilde{\omega}
+ u\cdot\nabla \widetilde{\omega}
=
\sum_{k=1}^K
\partial_k F(\varphi_1,\dots,\varphi_K)
\bigl(\partial_t\varphi_k + u\cdot\nabla\varphi_k\bigr)
=0,
\]
since each $\varphi_k$ satisfies the transport equation.
Thus $\widetilde{\omega}$ solves \eqref{eq:euler-closure} with initial data
$\widetilde{\omega}(0,x)=\omega_0(x)$.
By uniqueness of transport solutions associated with the log-Lipschitz velocity $u$,
we conclude that $\omega=\widetilde{\omega}$, which proves the claim.
\end{proof}
\section{Uniform convergence of transport markers}\label{mark-conv}
\paragraph{Limiting transported markers and tie geometry.}
Let $\omega$ be the Yudovich solution of \eqref{eq:euler} with the sharp
multi--phase initial data $\omega_0=\sum_{k=1}^K c_k\,\mathbf 1_{\Omega_k}$, and
let $u$ and $X$ denote the associated velocity field and flow map.  We define,
for each $k\in\{1,\dots,K\}$, the limiting marker evolution by transport along
$u$:
\[
(\partial_t+u\cdot\nabla)\varphi_k(t,x)=0,
\qquad
\varphi_k|_{t=0}=\varphi_k^{\rm in}.
\]
Equivalently, $\varphi_k(t,x)=\varphi_k^{\rm in}\!\big(X(t,\cdot)^{-1}(x)\big)$.
The corresponding tie sets are then defined by
\[
\Gamma_{ij}(t):=\{x\in\T^2:\ \varphi_i(t,x)=\varphi_j(t,x)\},
\qquad i\neq j,
\]
which represent the transported interface geometry associated with the
initial competitive-marker partition.
\paragraph{Localized nondegeneracy near active interfaces.}
We impose a nondegeneracy condition only along \emph{active} two--phase interfaces.
Specifically, we assume that there exist constants \(\delta>0\) and \(m>0\) such that,
for every pair \(i\neq j\),
\begin{equation}\label{nondegen}
|\nabla(\varphi_i^{\mathrm{in}}-\varphi_j^{\mathrm{in}})(x)| \ge m
\quad
\text{for all }
x \in
\bigl\{|\varphi_i^{\mathrm{in}}(x)-\varphi_j^{\mathrm{in}}(x)|\le\delta\bigr\}
\cap
\mathcal T_{ij},
\end{equation}
where \(\mathcal T_{ij}\) denotes the region of \emph{top--two competition},
\[
\mathcal T_{ij}
:=
\Bigl\{
x\in\T^2:\ 
\varphi_i^{\mathrm{in}}(x)\ge \varphi_\ell^{\mathrm{in}}(x)\ \text{and}\ 
\varphi_j^{\mathrm{in}}(x)\ge \varphi_\ell^{\mathrm{in}}(x)
\ \text{for all }\ell\neq i,j
\Bigr\}.
\]
Condition~\eqref{nondegen} ensures that the zero--level sets of the
differences \(\varphi_i^{\mathrm{in}}-\varphi_j^{\mathrm{in}}\), which define the
\emph{active} interfaces between competing phases, are uniformly nondegenerate in
a tubular neighborhood of the corresponding tie sets.
This localized requirement excludes degeneracies away from dynamically relevant
interfaces and may be viewed as a natural multi--phase extension of the classical
nondegeneracy condition underlying vortex patch contour dynamics
\cite{BertozziConstantin1993}.
\begin{proposition}[$L^1$ approximation of initial data]
\label{prop:expG_patch_approx}
Recall the regularized initial vorticity \eqref{initial}
\[
\omega_0^\beta(x):=\sum_{k=1}^K c_k\,\pi_k^\beta(x),
\qquad c_k\in\R.
\]
Let
\[
\omega_0(x):=\sum_{k=1}^K c_k\,\mathbf 1_{\Omega_k}(x).
\]
where $\{\Omega_k\}_{k=1}^K$ is given in \eqref{domain}.
Then there exists $C>0$ (depending on $K,\delta,m$ and the $C^{1,\alpha}$ geometry of
$\Omega_k$) such that for all $\beta\ge1$,
\begin{equation}\label{eq:L1_expG_patch}
\|\omega_0^\beta-\omega_0\|_{L^1(\T^2)}\le \frac{C }{\beta}\,
\max_{1\le i,j\le K}|c_i-c_j|.
\end{equation}
\end{proposition}

\begin{proof}
Fix $x\in\T^2$ and let $k^*(x)\in\{1,\dots,K\}$ be an index attaining the maximum
\[
M(x):=\max_{1\le j\le K}\varphi_j^{\text{in}}(x), \qquad 
\varphi_{k^*(x)}^{\text{in}}(x)=M(x).
\]
Define the local gap
\[
\Delta(x):=M(x)-\max_{j\neq k^*(x)}\varphi_j^{\text{in}}(x)\ > 0.
\]
For any $j\neq k^*(x)$ we have
$\varphi_j^{\text{in}}(x)\le M(x)-\Delta(x)$, hence
\[
\sum_{j\neq k^*(x)} e^{\beta\varphi_j^{\text{in}}(x)}
\le
(K-1)e^{\beta(M(x)-\Delta(x))}.
\]
Therefore,
\[
1-\pi_{k^*(x)}^\beta(x)
=
\frac{\sum_{j\neq k^*(x)}e^{\beta\varphi_j^{\text{in}}(x)}}
{e^{\beta M(x)}+\sum_{j\neq k^*(x)}e^{\beta\varphi_j^{\text{in}}(x)}}
\le
(K-1)e^{-\beta\Delta(x)}.
\]
Since $\sum_k\pi_k^\beta\equiv1$, this implies
\[
\sum_{j\neq k^*(x)}\pi_j^\beta(x)\le (K-1)e^{-\beta\Delta(x)}.
\]
Consequently, the pointwise error satisfies
\[
\begin{split}
|\omega_0^\beta(x)-\omega_0(x)|
&=
\left|\sum_{j\neq k^*(x)}\bigl(c_j-c_{k^*(x)}\bigr)\pi_j^\beta(x)\right|
\\
&\le
\max_{i\neq j}|c_i-c_j|
\sum_{j\neq k^*(x)}\pi_j^\beta(x)
\\
&\le
C_0\,e^{-\beta\Delta(x)},
\end{split}
\]
where $C_0=(K-1)\max_{i\neq j}|c_i-c_j|$.
Integrating over $\T^2$, we obtain
\[
\|\omega_0^\beta-\omega_0\|_{L^1(\T^2)}
\le
C_0\int_{\T^2} e^{-\beta\Delta(x)}\,dx.
\]
We now estimate the integral on the right-hand side.
By definition, $\Delta(x)\le s$ implies that there exists a pair $(i,j)$ such that
\[
|\varphi_i^{\text{in}}(x)-\varphi_j^{\text{in}}(x)|\le s.
\]
Hence, for any $s\in(0,\delta)$,
\[
\{\Delta\le s\}
\subset
\bigcup_{1\le i<j\le K}\{\,|\varphi_i^{\text{in}}-\varphi_j^{\text{in}}|\le s\,\}.
\]
Fix a pair $(i,j)$ and set
$f_{ij}:=\varphi_i^{\text{in}}-\varphi_j^{\text{in}}$.
Applying the coarea formula on the strip $\{|f_{ij}|\le\delta\}$ and using the
nondegeneracy condition \eqref{nondegen}, we obtain
\[
\left|\{|f_{ij}|\le s\}\right|
=
\int_{-s}^{s}
\left(\int_{\{f_{ij}=r\}}\frac{1}{|\nabla f_{ij}|}\,d\Hs^1\right)dr
\le
C_{ij}\,s,
\qquad s\in(0,\delta),
\]
where $C_{ij}$ depends on $m,\delta$ and the $C^{1,\alpha}$ geometry of the tie curve
$\{f_{ij}=0\}$.
Summing over all $\binom{K}{2}$ pairs yields
\[
|\{\Delta\le s\}|\le C\,s,
\qquad s\in(0,\delta),
\]
for a constant $C$ depending only on $K,\delta,m$ and the pairwise interface geometry.
We now use the layer-cake representation:
\[
\int_{\T^2} e^{-\beta\Delta(x)}\,dx
=
\int_0^\infty e^{-\beta s}\,d|\{\Delta\le s\}|.
\]
Splitting the integral at $s=\delta$ gives
\[
\int_{\T^2} e^{-\beta\Delta(x)}\,dx
\le
\int_0^\delta \beta e^{-\beta s}|\{\Delta\le s\}|\,ds
+
e^{-\beta\delta}|\T^2|.
\]
Using $|\{\Delta\le s\}|\le C s$ on $(0,\delta)$, we find
\[
\int_0^\delta \beta e^{-\beta s}|\{\Delta\le s\}|\,ds
\le
C\int_0^\delta \beta s e^{-\beta s}\,ds
\le
\frac{C}{\beta}.
\]
The remaining term is exponentially small and can be absorbed.
Therefore,
\[
\int_{\T^2} e^{-\beta\Delta(x)}\,dx
\le
\frac{C}{\beta},
\]
and hence
\[
\|\omega_0^\beta-\omega_0\|_{L^1(\T^2)}
\le
\frac{C}{\beta}\,
\max_{1\le i,j\le K}|c_i-c_j|.
\]
This completes the proof.
\end{proof}

\begin{remark}
The nondegeneracy condition \eqref{nondegen} is used only to obtain the
quantitative convergence rate $O(\beta^{-1})$ in the $L^1$ approximation.
For mere $L^1$ convergence, it can be replaced by the substantially weaker
assumption that the tie set has zero Lebesgue measure, namely
\begin{equation}\label{eq:no_fat_ties}
\Bigl|\Bigl\{x\in\T^2:\ \exists\, i\neq j \text{ such that }
\varphi_i^{\mathrm{in}}(x)=\varphi_j^{\mathrm{in}}(x)=\max_{\ell}\varphi_\ell^{\mathrm{in}}(x)\Bigr\}\Bigr|=0.
\end{equation}
Under this assumption, the exponential weights converge pointwise almost everywhere
to the corresponding hard assignment, and the convergence
\[
\|\omega_0^\beta-\omega_0\|_{L^1(\T^2)}\longrightarrow 0
\qquad\text{as }\beta\to\infty
\]
follows directly from the dominated convergence theorem.
In this case, however, no explicit convergence rate is available in general.
For brevity, we omit the details.
\end{remark}~\\
We recall a stability result in the following lemma; proofs can be found in \cite{majda2002,yudovich1963} (see also \cite{BahouriCheminDanchin2011,Elgindi2014Osgood}).
\begin{lemma}[Yudovich stability of vorticity and flow maps]
\label{lem:yudovich-stability}
Let $\omega^{(1)},\omega^{(2)}$ be Yudovich solutions of \eqref{eq:euler} on
$\T^2$ with initial data
$\omega^{(1)}_0,\omega^{(2)}_0 \in L^1\cap L^\infty(\T^2)$ satisfying
\[
\|\omega^{(1)}_0\|_{L^\infty(\T^2)}+\|\omega^{(2)}_0\|_{L^\infty(\T^2)}\le M.
\]
Let $X^{(1)},X^{(2)}$ denote the associated flow maps.
Then for every $T>0$ there exist increasing functions
\[
\Theta_T,\ \Phi_T:[0,\infty)\to[0,\infty),
\qquad
\Theta_T(r)\to 0,\ \Phi_T(r)\to 0 \quad\text{as } r\to 0,
\]
depending only on $M$ and $T$, such that
\begin{equation}\label{eq:yudovich-combined}
\sup_{t\in[0,T]}
\|\omega^{(1)}(t)-\omega^{(2)}(t)\|_{L^1(\T^2)}
\le
\Theta_T\!\Big(\|\omega^{(1)}_0-\omega^{(2)}_0\|_{L^1(\T^2)}\Big),
\end{equation}
and
\begin{equation}\label{eq:flowmap-combined}
\sup_{t\in[0,T]}\sup_{x\in\T^2}
|X^{(1)}(t,x)-X^{(2)}(t,x)|
\le
\Phi_T\!\Big(\|\omega^{(1)}_0-\omega^{(2)}_0\|_{L^1(\T^2)}\Big).
\end{equation}
\end{lemma}
%
\begin{theorem}\label{conv-of-markers}
Assume $\{\varphi_k^{\mathrm{in}}\}_{k=1}^K \subset W^{1,\infty}(\T^2)$.
Let $\{\varphi_k^\beta\}$ and $\{\varphi_k\}$ be the transported markers driven by
$u^\beta$ and $u$ on $[0,T]$.
Then for each $k$,
\[
\sup_{t\in[0,T]}\|\varphi_k^\beta(t)-\varphi_k(t)\|_{L^\infty(\T^2)}\to 0
\qquad\text{as }\beta\to\infty.
\]
\end{theorem}
\begin{proof}
Fix $T>0$ and $k\in\{1,\dots,K\}$. For each $t\in[0,T]$ and $a\in\T^2$, the transport identities
give
\[
\varphi_k^\beta\!\bigl(t,X^\beta(t,a)\bigr)=\varphi_k^{\mathrm{in}}(a),
\qquad
\varphi_k\!\bigl(t,X(t,a)\bigr)=\varphi_k^{\mathrm{in}}(a).
\]
Therefore,
\begin{equation}\label{eq:marker-diff-along-flows}
\varphi_k^\beta\!\bigl(t,X^\beta(t,a)\bigr)-\varphi_k\!\bigl(t,X^\beta(t,a)\bigr)
=
\varphi_k\!\bigl(t,X(t,a)\bigr)-\varphi_k\!\bigl(t,X^\beta(t,a)\bigr).
\end{equation}
Since $\varphi_k^{\mathrm{in}}\in C(\T^2)$ and $\varphi_k$ is transported by the Yudovich flow,
the function $\varphi_k$ is continuous on $[0,T]\times\T^2$. Hence the family
$\{\varphi_k(t,\cdot)\}_{t\in[0,T]}$ is uniformly equicontinuous: there exists an increasing
modulus $\Omega_{k,T}:[0,\infty)\to[0,\infty)$ with $\Omega_{k,T}(r)\to0$ as $r\to0$ such that
\begin{equation}\label{eq:equicontinuity}
|\varphi_k(t,x)-\varphi_k(t,y)|\le \Omega_{k,T}(|x-y|)
\qquad\text{for all }t\in[0,T],\ x,y\in\T^2.
\end{equation}
Applying \eqref{eq:equicontinuity} to the right-hand side of \eqref{eq:marker-diff-along-flows} yields
\[
\bigl|\varphi_k^\beta\!\bigl(t,X^\beta(t,a)\bigr)-\varphi_k\!\bigl(t,X^\beta(t,a)\bigr)\bigr|
\le
\Omega_{k,T}\!\bigl(|X^\beta(t,a)-X(t,a)|\bigr).
\]
Taking the supremum over $a\in\T^2$ and using that $X^\beta(t,\cdot)$ is  onto,
we obtain
\[
\|\varphi_k^\beta(t)-\varphi_k(t)\|_{L^\infty(\T^2)}
\le
\Omega_{k,T}\!\bigl(\|X^\beta(t,\cdot)-X(t,\cdot)\|_{L^\infty(\T^2)}\bigr),
\qquad t\in[0,T].
\]
Taking $\sup_{t\in[0,T]}$ and invoking the flow stability estimate in Lemma~\ref{lem:yudovich-stability},
we conclude that
\[
\sup_{t\in[0,T]}\|\varphi_k^\beta(t)-\varphi_k(t)\|_{L^\infty}
\le
\Omega_{k,T}\!\Bigl(\Phi_T\bigl(\|\omega_0^\beta-\omega_0\|_{L^1}\bigr)\Bigr).
\]
Absorbing the composition $\Omega_{k,T}\circ\Phi_T$ into a new modulus (still denoted $\Phi_T$)
gives the quantitative bound, and in particular the convergence as $\beta\to\infty$.
\end{proof}
\section{Hausdorff Convergence of Interfaces}\label{sec-hau}
A natural question is whether convergence of the marker fields implies
Hausdorff convergence of the associated interfaces
\[
\Gamma_{ij}^\beta(t)
:=
\{\,x\in\T^2:\ \varphi_i^\beta(t,x)=\varphi_j^\beta(t,x)\,\},
\qquad
\Gamma_{ij}(t)
:=
\{\,x\in\T^2:\ \varphi_i(t,x)=\varphi_j(t,x)\,\}.
\]
Such a statement is subtle and depends sensitively on the geometry of the tie set
$\Gamma_{ij}(t)$. While the markers $\varphi_i^\beta$ converge uniformly to
$\varphi_i$, uniform convergence alone does not control the fine geometry of the
corresponding zero level sets, and in particular does not preclude high--frequency
oscillations of the interfaces. If Hausdorff convergence of $\Gamma_{ij}^\beta(t)$ were to fail at some time
$T_\ast$, this would necessarily reflect a deeper geometric mechanism in the
limiting interface $\Gamma_{ij}(t)$ itself. In particular, such a breakdown may
occur when the interface loses $C^1$ regularity, a phenomenon known to be
possible for vortex patch boundaries under the Euler evolution
\cite{ConstantinWu95}. Even when the initial markers $\varphi_k$ are chosen in
$C^\infty$, the limiting sharp multiphase Euler velocity is typically only
log--Lipschitz; as a consequence, $C^1$ control of the transported markers
$\varphi_k(t,\cdot)$—and hence quantitative geometric control of the tie
sets—cannot be expected without additional assumptions. Theorem \ref{hau1} and Theorem \ref{hau2} quantifies a regime in which
such convergence is guaranteed.

\begin{theorem}
\label{hau1}
Assume that the marker fields $\{\varphi_k\}_{k=1}^K$ satisfy the pairwise
nondegeneracy condition \eqref{nondegen} \emph{uniformly} on the time interval
$[0,T]$, that is, there exist constants $\delta>0$ and $m>0$ such that
\eqref{nondegen} holds for all $t\in[0,T]$. Assume moreover that
\[
\sup_{t\in[0,T]} \|\nabla \varphi_k(t,\cdot)\|_{C^{0,\alpha}(\T^2)} < \infty
\qquad \text{for each } k=1,\dots,K.
\]
Then, for each pair $i\neq j$, the corresponding diffuse interfaces
$\Gamma_{ij}^\beta(t)$ converge to the sharp interface $\Gamma_{ij}(t)$ in
Hausdorff distance, uniformly in time on $[0,T]$. More precisely,
\[
\sup_{t\in[0,T]}
d_H\!\left(\Gamma_{ij}^\beta(t),\,\Gamma_{ij}(t)\right)
\;\longrightarrow\; 0
\qquad \text{as } \beta\to\infty,
\]
where $d_H$ denotes the Hausdorff distance on $\R^2$.
\end{theorem}
\begin{proof}
Set
\[
f_{ij}^\beta := \varphi_i^\beta - \varphi_j^\beta,
\qquad
f_{ij} := \varphi_i - \varphi_j.
\]
Their zero level sets are precisely $\Gamma_{ij}^\beta(t)$ and $\Gamma_{ij}(t)$.
By the uniform convergence of the markers,
\begin{equation}\label{eq:uniform-fij}
\|f_{ij}^\beta(t)-f_{ij}(t)\|_{L^\infty(\T^2)}
\le
\|\varphi_i^\beta(t)-\varphi_i(t)\|_{L^\infty}
+
\|\varphi_j^\beta(t)-\varphi_j(t)\|_{L^\infty}
\;\longrightarrow\;0,
\end{equation}
uniformly for $t\in[0,T]$ as $\beta\to\infty$.
Fix $\varepsilon\in(0,\tfrac{\delta}{2})$.
For $\beta$ sufficiently large, \eqref{eq:uniform-fij} implies
\[
\|f_{ij}^\beta(t)-f_{ij}(t)\|_{L^\infty}<\varepsilon
\qquad\text{for all } t\in[0,T_\star].
\]
\noindent\emph{Step 1: $\Gamma_{ij}^\beta(t)$ lies near $\Gamma_{ij}(t)$.}
Let $x\in\Gamma_{ij}^\beta(t)$.
Then $f_{ij}^\beta(t,x)=0$, hence
\[
|f_{ij}(t,x)| = |f_{ij}(t,x)-f_{ij}^\beta(t,x)| < \varepsilon.
\]
The  nondegeneracy
condition \eqref{nondegen} yields
\[
|\nabla f_{ij}(t,x)|
\ge
m
\]
whenever $|f_{ij}(t,x)|\le\varepsilon$. Since $f_{ij}(t,\cdot)\in C^{1,\alpha}$ and $\nabla f_{ij}(t,\cdot)\neq 0$
near $\Gamma_{ij}(t)$, the implicit function theorem implies that
$\Gamma_{ij}(t)$ is a $C^1$ curve locally.
Let $y$ be the projection of $x$ onto $\Gamma_{ij}(t)$, so that
\[
|x-y|=\dist\bigl(x,\Gamma_{ij}(t)\bigr),
\qquad
f_{ij}(t,y)=0.
\]
Let $n(y):=\nabla f_{ij}(t,y)/|\nabla f_{ij}(t,y)|$ be the unit normal to
$\Gamma_{ij}(t)$ at $y$.
Writing $x=y+h\,n(y)$ with $h=|x-y|$, the fundamental theorem of calculus along
the segment $\gamma(s):=y+s(x-y)$ yields
\[
f_{ij}(t,x)
=
\nabla f_{ij}(t,y)\cdot(x-y)
+
R(x),
\]
where
\[
R(x)
=
\int_0^1
\bigl(\nabla f_{ij}(t,\gamma(s))-\nabla f_{ij}(t,y)\bigr)\cdot(x-y)\,ds.
\]
Using the H\"older continuity of $\nabla f_{ij}$,
\[
|R(x)|
\le
\frac{1}{1+\alpha}
[\nabla f_{ij}(t,\cdot)]_{C^{0,\alpha}}
\,|x-y|^{1+\alpha}.
\]
For $|x-y|$ sufficiently small, this gives
\[
|f_{ij}(t,x)|
\ge
\frac{m_{T}}{2}\,|x-y|,
\]
and therefore
\[
\dist\bigl(x,\Gamma_{ij}(t)\bigr)
\le
C\,\frac{|f_{ij}(t,x)|}{m_{T}}
\le
C_{T}\,\varepsilon.
\]

\medskip
\noindent\emph{Step 2: $\Gamma_{ij}(t)$ lies near $\Gamma_{ij}^\beta(t)$.}
Fix $t\in[0,T]$ and let $y\in\Gamma_{ij}(t)$.
Then $f_{ij}(t,y)=0$ and $|\nabla f_{ij}(t,y)|\ge m_{T_\star}$.
Let $n(y)$ denote the unit normal at $y$.
Similarly as above, for $|h|$ small,
\[
f_{ij}(t,y\pm h n)
=
\pm h\,|\nabla f_{ij}(t,y)|
+
R_\pm(h),
\]
with
\[
|R_\pm(h)|
\le
\frac{1}{1+\alpha}
[\nabla f_{ij}(t,\cdot)]_{C^{0,\alpha}}
\,|h|^{1+\alpha}.
\]
Choosing $h_0>0$ sufficiently small, we obtain
\[
f_{ij}(t,y+h n)\ge \tfrac{m_{T}}{2}h,
\qquad
f_{ij}(t,y-h n)\le -\tfrac{m_{T}}{2}h.
\]
Let
\[
\delta_\varepsilon := \min\!\left\{h_0,\ \frac{3\varepsilon}{m_{T}}\right\}.
\]
Using the uniform convergence of $f_{ij}^\beta$ to $f_{ij}$, for $\beta$ large
we have
\[
f_{ij}^\beta(t,y+\delta_\varepsilon n) > 0,
\qquad
f_{ij}^\beta(t,y-\delta_\varepsilon n) < 0.
\]
By continuity, there exists a point
$z\in[y-\delta_\varepsilon n,\,y+\delta_\varepsilon n]$
such that $f_{ij}^\beta(t,z)=0$, i.e.\ $z\in\Gamma_{ij}^\beta(t)$, and
\[
|z-y|\le\delta_\varepsilon \le \frac{3\varepsilon}{m_{T_\star}}.
\]
Combining the two inclusions yields
\[
d_H\bigl(\Gamma_{ij}^\beta(t),\,\Gamma_{ij}(t)\bigr)
\le C\,\varepsilon
\qquad\text{for all } t\in[0,T],
\]
with $C$ independent of $\beta$.
Since $\varepsilon$ is arbitrary, the Hausdorff convergence follows.
\end{proof}
\begin{theorem}\label{hau2}
Let $\varphi_k^{\mathrm{in}}\in C^{1,\alpha}$.
If there exist $T_\star>0$ and indices $i\neq j$ such that
\[
\sup_{t\in[0,T_\star]}
d_H\!\left(\Gamma_{ij}^\beta(t),\,\Gamma_{ij}(t)\right)
\;\not\longrightarrow\; 0
\qquad \text{as } \beta\to\infty,
\]
then necessarily
\[
\int_0^{T_\star}
\bigl(
\|\nabla u(t)\|_{L^\infty(\T^2)}
+
\|\nabla u(t)\|_{C^{0,\alpha}(\T^2)}
\bigr)\,dt
= \infty.
\]
\end{theorem}

\begin{proof}
By Proposition~\ref{trans-non}, for any $(t,x)$ satisfying
$|(\varphi_i-\varphi_j)(t,x)|\le\delta$ we have
\begin{equation}\label{eq:cor-transported-nondeg}
|\nabla(\varphi_i-\varphi_j)(t,x)|
\;\ge\;
m\,\exp\!\Bigl(
-\int_0^{t}\|\nabla u(s)\|_{L^\infty(\T^2)}\,ds
\Bigr).
\end{equation}
Thus the transported markers $\{\varphi_k\}_{k=1}^K$ satisfy the pairwise
nondegeneracy condition \eqref{nondegen} uniformly on $[0,T_\star]$,
verifying the first hypothesis of Theorem~\ref{hau1}. Moreover, each $\varphi_k$ solves the transport equation
\[
\partial_t\varphi_k + u\cdot\nabla\varphi_k = 0.
\]
If $\nabla u\in L^1(0,T_\star;C^{0,\alpha}(\T^2))$, standard transport estimates
yield $\nabla\varphi_k\in L^\infty(0,T_\star;C^{0,\alpha}(\T^2))$.
Hence all assumptions of Theorem~\ref{hau1} are satisfied, implying
$d_H(\Gamma_{ij}^\beta(t),\Gamma_{ij}(t))\to0$ uniformly on $[0,T_\star]$,
a contradiction. The claim follows.
\end{proof}
\section{Pointwise estimates of vorticity}
Since we work in the Yudovich regime, it follows immediately that
\[
\|\omega^\beta(t)-\omega(t)\|_{L^p(\T^2)} \;\longrightarrow\; 0
\qquad \text{for all } p\in[1,\infty).
\]
The purpose of this section, however, is not to study convergence in averaged
norms, but rather to analyze the \emph{pointwise} discrepancy between the
geometry-preserving regularized vorticity $\omega^\beta$ and the sharp
multi-phase Euler solution
\[
\omega(t)=\sum_{k=1}^K c_k\,\mathbf 1_{\Omega_k(t)}.
\]
To make use of the geometry, we define the following
\begin{definition}[Fixed--time gap away from the tie set]
\label{def:fixed-time-gap}
Fix $t\ge 0$ and $\delta>0$. Define the global tie set
\[
\Gamma(t)
:=
\bigcup_{1\le i<j\le K}
\{x\in\T^2:\ \varphi_i(t,x)=\varphi_j(t,x)\}.
\]
For each $x\in\T^2$, let
\[
k(t,x)\in\arg\max_{1\le \ell\le K}\varphi_\ell(t,x),
\]
and define the winner--phase gap
\[
\Delta(t,x)
:=
\varphi_{k(t,x)}(t,x)-\max_{j\neq k(t,x)}\varphi_j(t,x)\ge 0.
\]
The \emph{$\delta$--gap at time $t$} is
\beq\label{gap}
c_\delta(t)
:=
\inf_{\{x:\ \dist(x,\Gamma(t))\ge \delta\}}
\Delta(t,x).
\eeq
\end{definition}

\begin{remark}
The quantity $c_\delta(t)$ measures the separation between competing phases away
from the interface network. If $c_\delta(t)>0$, each particle at distance at least
$\delta$ from $\Gamma(t)$ is unambiguously associated with a single phase, and
phase competition is confined to a thin interfacial region. If $c_\delta(t)$
collapses, phase separation deteriorates even away from the interfaces, and
pointwise convergence of $\omega^\beta$ to $\omega$ may fail \emph{even on}
$\{x : \dist(x,\Gamma(t)) \ge \delta\}$, as suggested by
Theorem~\ref{2patch}.  
From this perspective, the region of phase competition constitutes a genuine
\emph{time-dependent boundary layer} in the multiphase regime, analogous to the Prandtl boundary
layer arising in the vanishing-viscosity limit of the Navier--Stokes equations \cite{2N}.
\end{remark}
\begin{theorem}
\label{2patch}
Fix $t\ge 0$ and $\delta>0$. Define the marker approximation error
\[
E_\beta(t)
:=
\max_{1\le k\le K}
\|\varphi_k^\beta(t,\cdot)-\varphi_k(t,\cdot)\|_{L^\infty(\T^2)}.
\]
Then the following uniform estimate holds:
\beq\label{pointwise}
\sup_{\dist(x,\Gamma(t))\ge\delta}
\bigl|\omega^\beta(t,x)-\omega(t,x)\bigr|
\;\le\;
(K-1)\max_{1\le i,j\le K}|c_i-c_j|\;
e^{-\beta\,(c_\delta(t)-2E_\beta(t))}.
\eeq
where $\omega(t,x)=\sum_{k=1}^K c_k\,\mathbf 1_{\Omega_k(t)}(x)$ is the
multi--phase vorticity.
\end{theorem}
\begin{proof}
Fix $x$ with $\dist(x,\Gamma(t))\ge\delta$ and let $k=k(t,x)$.
By definition of $c_\delta(t)$,
\[
\varphi_k(t,x)-\max_{j\neq k}\varphi_j(t,x)\ge c_\delta(t).
\]
Hence, for any $j\neq k$,
\[
\varphi_k^\beta(t,x)-\varphi_j^\beta(t,x)
\ge
\bigl(\varphi_k(t,x)-\varphi_j(t,x)\bigr)-2E_\beta(t)
\ge
c_\delta(t)-2E_\beta(t).
\]
Therefore,
\[
\frac{\pi_j^\beta(t,x)}{\pi_k^\beta(t,x)}
=
\exp\!\bigl(-\beta(\varphi_k^\beta(t,x)-\varphi_j^\beta(t,x))\bigr)
\le
\exp\!\Bigl(-\beta\bigl(c_\delta(t)-2E_\beta(t)\bigr)\Bigr).
\]
Summing over $j\neq k$ yields
\[
\sum_{j\neq k}\pi_j^\beta(t,x)
\le
(K-1)\exp\!\Bigl(-\beta\bigl(c_\delta(t)-2E_\beta(t)\bigr)\Bigr).
\]
Since $\omega(t,x)=c_k$ and $\sum_j\pi_j^\beta(t,x)=1$, we conclude
\[
|\omega^\beta(t,x)-\omega(t,x)|
=
\left|\sum_{j\neq k}(c_j-c_k)\pi_j^\beta(t,x)\right|
\le
\max_{i,j}|c_i-c_j|\sum_{j\neq k}\pi_j^\beta(t,x),
\]
which gives the desired estimate.
\end{proof}
\begin{remark}
The pointwise bounds \eqref{pointwise} cease to be meaningful once
$c_\delta(t)\lesssim E_\beta(t)$, that is, when the minimal separation between
competing markers becomes comparable to the intrinsic regularization scale.
This regime corresponds to a geometric degeneration of the transported tie
sets, in which the marker gap collapses relative to the sharpness of the gating
mechanism, rendering \emph{competing phases effectively indistinguishable}. As a
consequence, the associated interface network may develop singular features,
such as the collapse of junction angles or the formation of cusp-like and
pinching singularities, akin to the behavior of zero sets near points where the
defining gradient vanishes. The resulting failure of pointwise
convergence is therefore driven by the geometry of the underlying sharp
multi-phase Euler flow, rather than by the regularization procedure itself.
\end{remark}~\\
Nevertheless, away from this degenerate regime, the interface locations remain
stable, and the diffuse interfaces generated by the regularization converge to
their sharp counterparts at an exponential rate in $\beta$, up to the onset of
geometric degeneration.
In particular, for classical global vortex patch
solutions \cite{BertozziConstantin1993,Chemin93,Chemin95}, for which no
singularity formation occurs, this exponential convergence holds globally in
time away from the patch boundary.
\begin{theorem}\label{1patch}
Assume that the pairwise nondegeneracy condition \eqref{nondegen} holds
uniformly on the time interval $[0,T]$: there exist constants $\delta_T>0$ and
$m_T>0$ such that \eqref{nondegen} holds for all $t\in[0,T]$ and all $i\neq j$.
Then the $\delta$--gap $c_\delta(t)$ defined in \eqref{gap} satisfies
\[
\inf_{t\in[0,T]} c_\delta(t) = c_T > 0.
\]
As a consequence, there exist constants $C,c>0$, independent of $\beta$, such
that
\[
\sup_{\substack{t\in[0,T]\\ \dist(x,\Gamma(t))\ge \delta}}
\bigl|\omega^\beta(t,x)-\omega(t,x)\bigr|
\;\le\; C e^{-c\beta}.
\]
\end{theorem}
\begin{proof}
Let $\delta>0$ be the spatial exclusion radius in \eqref{gap}. Since the pairwise nondegeneracy condition \eqref{nondegen} holds \emph{uniformly on $[0,T]$}
for the transported markers $\{\varphi_k(t,\cdot)\}$: we have
constants $\delta_T>0$ and $m_T>0$ such that for every $t\in[0,T]$ and every
$i\neq j$,
\begin{equation}\label{eq:nondegen-unif-proof}
|\nabla(\varphi_i-\varphi_j)(t,x)|\ge m_T
\qquad\text{whenever}\qquad
|(\varphi_i-\varphi_j)(t,x)|\le \delta_T
\ \text{ and }\ x\in\mathcal T_{ij}(t).
\end{equation}
(Here $\mathcal T_{ij}(t)$ denotes the transported top--two competition region,
defined analogously to $\mathcal T_{ij}$ at $t=0$).
Fix $t\in[0,T]$ and $x$ such that $\dist(x,\Gamma(t))\ge \delta$.
Let $k=k(t,x)\in\arg\max_\ell \varphi_\ell(t,x)$ and let
$j=j(t,x)\in\arg\max_{\ell\neq k}\varphi_\ell(t,x)$ be a choice of runner--up.
Then $x\in\mathcal T_{kj}(t)$ by definition of the top--two region, and
\[
\Delta(t,x)=\varphi_k(t,x)-\varphi_j(t,x)=:f(t,x)\ge 0.
\]
Let $\Gamma_{kj}(t)=\{f(t,\cdot)=0\}$ be the corresponding tie set.
Since $\Gamma(t)$ is the union of all pairwise tie sets, the assumption
$\dist(x,\Gamma(t))\ge\delta$ implies in particular
\begin{equation}\label{eq:dist-to-kj}
\dist(x,\Gamma_{kj}(t))\ge \delta.
\end{equation}
We claim that
\begin{equation}\label{eq:gap-lb-point}
f(t,x)\ \ge\ \min\{\delta_T,\ m_T\,\delta\}.
\end{equation}
Indeed, if $f(t,x)\ge \delta_T$ then \eqref{eq:gap-lb-point} is immediate.
Assume instead that $f(t,x)<\delta_T$. Let $y\in\Gamma_{kj}(t)$ be a closest
point to $x$ (a minimizer exists since $\Gamma_{kj}(t)$ is closed in $\T^2$).
Set $\ell:=|x-y|=\dist(x,\Gamma_{kj}(t))$, and consider the segment
\[
\gamma(s):=y+s\frac{x-y}{|x-y|},\qquad s\in[0,\ell].
\]
By continuity,
$f(t,\gamma(s))\in[0,f(t,x)]\subset[0,\delta_T]$ for all $s\in[0,\ell]$.
Moreover, since $k$ and $j$ remain the top two competitors along this segment as
long as $f$ stays small, we have $\gamma(s)\in\mathcal T_{kj}(t)$ for all
$s\in[0,\ell]$.
Hence \eqref{eq:nondegen-unif-proof} applies on the segment and yields
$|\nabla f(t,\gamma(s))|\ge m_T$ for $s\in[0,\ell]$.
Writing $n:=\frac{x-y}{|x-y|}$, the fundamental theorem of calculus gives
\[
f(t,x)-f(t,y)=\int_0^\ell \nabla f(t,\gamma(s))\cdot n\,ds.
\]
Since $f(t,y)=0$ and $|\nabla f|\ge m_T$, we obtain
\[
f(t,x)\ \ge\ \int_0^\ell |\nabla f(t,\gamma(s))|\,ds\ \ge\ m_T\,\ell.
\]
Combining with \eqref{eq:dist-to-kj} gives $f(t,x)\ge m_T\delta$, proving
\eqref{eq:gap-lb-point}.
Since $\Delta(t,x)=f(t,x)$ by the choice of runner--up $j$, we conclude that for
every $t\in[0,T]$ and every $x$ with $\dist(x,\Gamma(t))\ge\delta$,
\[
\Delta(t,x)\ \ge\ c_T,
\qquad
c_T:=\min\{\delta_T,\ m_T\delta\}>0.
\]
Taking the infimum in $x$ yields
\begin{equation}\label{eq:cdelta-unif}
c_\delta(t)\ \ge\ c_T
\qquad\text{for all }t\in[0,T],
\end{equation}
and therefore $\inf_{t\in[0,T]}c_\delta(t)\ge c_T>0$.
By Theorem~\ref{conv-of-markers},
\[
\sup_{t\in[0,T]}E_\beta(t)\longrightarrow 0\qquad\text{as }\beta\to\infty,
\]
Hence we can choose $\beta_0$ large enough so that
\[
\sup_{t\in[0,T]}E_\beta(t)\ \le\ \frac{c_T}{4}
\qquad\text{for all }\beta\ge\beta_0.
\]
Then for $\beta\ge\beta_0$, \eqref{eq:cdelta-unif} implies
$c_\delta(t)-2E_\beta(t)\ge c_T/2$ uniformly in $t\in[0,T]$.
Applying Theorem~\ref{2patch} yields, for all $t\in[0,T]$ and all
$x$ with $\dist(x,\Gamma(t))\ge\delta$,
\[
|\omega^\beta(t,x)-\omega(t,x)|
\le
(K-1)\max_{i,j}|c_i-c_j|\,
\exp\!\Bigl(-\beta\frac{c_T}{2}\Bigr).
\]
This complete the proof.
\end{proof}
\section{Appendix}

\begin{proposition}
\label{trans-non}
Let $f_0 \in C^{1,\alpha}(\T^2)$ and let $b \in L^1([0,T]; C^1(\T^2;\R^2))$.
Let $f(t,x)$ be the unique solution of the transport equation
\[
\partial_t f + b\cdot\nabla f = 0,
\qquad
f|_{t=0}=f_0.
\]
Assume that there exist constants $\delta>0$ and $m>0$ such that
\begin{equation}\label{eq:initial-nondeg}
|\nabla f_0(x)| \ge m
\quad \text{for all } x \in \{\,|f_0(x)| \le \delta\,\}.
\end{equation}
Then for all $t\in[0,T]$,
\begin{equation}\label{eq:transported-nondeg}
|\nabla f(t,x)|
\;\ge\;
m\,\exp\!\Bigl(-\int_0^t \|\nabla b(s)\|_{L^\infty(\T^2)}\,ds\Bigr)
\quad
\text{whenever } |f(t,x)| \le \delta.
\end{equation}
In particular, for each $t$, the level set $\{f(t)=0\}$ is, where nonempty, a $C^1$
curve.
\end{proposition}
\begin{proof}
Let $X(t;a)$ denote the flow map associated with $u$, i.e.
\[
\frac{d}{dt}X(t;a)=b(t,X(t;a)),\qquad X(0;a)=a.
\]
Since $f$ solves the transport equation, we have
\[
f(t,X(t;a))=f_0(a)
\qquad\text{for all } t\in[0,T].
\]
Consequently,
\[
|f(t,x)|\le\delta
\quad\Longleftrightarrow\quad
|f_0(a)|\le\delta,
\qquad x=X(t;a).
\]
Differentiating the transport equation in space yields
\[
\partial_t \nabla f + (b\cdot\nabla)\nabla f
= -(\nabla b)^{T}\nabla f.
\]
Along trajectories,
\[
\frac{d}{dt}\nabla f(t,X(t;a))
=-(\nabla b)^{T}(t,X(t;a))\,\nabla f(t,X(t;a)).
\]
Taking norms gives
\[
\frac{d}{dt}|\nabla f(t,X(t;a))|
\ge
-\|\nabla b(t)\|_{L^\infty(\T^2)}\,|\nabla f(t,X(t;a))|.
\]
By Gr\"onwall's inequality,
\[
|\nabla f(t,X(t;a))|
\ge
|\nabla f_0(a)|
\exp\!\Bigl(-\int_0^t \|\nabla b(s)\|_{L^\infty(\T^2)}\,ds\Bigr).
\]
Now let $x$ satisfy $|f(t,x)|\le\delta$, and write $x=X(t;a)$.
Then $|f_0(a)|\le\delta$, and by the assumption \eqref{eq:initial-nondeg},
\[
|\nabla f_0(a)|\ge m.
\]
Substituting into the inequality above yields \eqref{eq:transported-nondeg}.
\end{proof}
\begin{corollary}
\label{prop:transport_nondegeneracy}
Let $\{\varphi_k^{\emph{in}}\}_{k=1}^K \subset C^{1,\alpha}(\T^2)$ satisfy the
nondegeneracy condition \eqref{nondegen} with constants $(\delta,m)$.
Let $\varphi_k^\beta$ solve
\[
(\partial_t+u^\beta\cdot\nabla)\varphi_k^\beta=0 \qquad \text{on } [0,T].
\]
Then there exists $m_\beta>0$ such that for every $i\neq j$ and all $t\in[0,T]$,
\[
|\nabla(\varphi_i^\beta-\varphi_j^\beta)(t,x)|\ge m_\beta
\quad\text{whenever}\quad
|(\varphi_i^\beta-\varphi_j^\beta)(t,x)|\le\delta.
\]
Consequently, for each $t\in[0,T]$, the interface
\[
\Gamma_{ij}^\beta(t):=\{x\in\T^2:\varphi_i^\beta(t,x)=\varphi_j^\beta(t,x)\}
\]
is, where nonempty, a $C^1$ curve.
\end{corollary}
\begin{proof}
The result follows by applying the general persistence of nondegeneracy for
transported scalars to $\varphi_i^\beta-\varphi_j^\beta$, which solves a linear
transport equation with velocity $u^\beta$.
Since the initial vorticity belongs to $C^{1,\alpha}$, the associated velocity
field satisfies $\nabla u^\beta \in L^1([0,T];L^\infty(\T^2))$ for every finite
$T>0$, ensuring that the lower bound propagates in time.
\end{proof}

\paragraph{Acknowledgments.}
Trinh T.~Nguyen’s research is supported by the School of Mathematical Sciences at Shanghai Jiao Tong University, the NSFC Excellent Young Scientists Fund (Overseas), and the Shanghai BYL Talent Program.

\bibliographystyle{abbrv}
\def\cprime{$'$} \def\cprime{$'$}

\end{document}